\documentclass{amsart}

\usepackage{ifxetex}
\ifxetex
\usepackage{fontspec}
\usepackage{unicode-math}
\defaultfontfeatures{Mapping=tex-text}
\else
\usepackage[utf8]{inputenc}
\usepackage[T1]{fontenc}
\usepackage[charter]{mathdesign}
\fi
\usepackage{microtype}

\usepackage{amsmath,amsthm}

\usepackage{hyperref}
\ifxetex
\hypersetup{unicode=true}
\fi

\hypersetup{
bookmarks=true,
pdfpagemode=UseNone,
pdfstartview=FitH,
pdfdisplaydoctitle=true,
pdflang=en-US,
pdfborder={0 0 0}, 
pdftitle={Typical operators admit common cyclic vectors},
pdfauthor={Pavel Zorin-Kranich},
pdfsubject={Mathematics Subject Classiﬁcation (2010): Primary 47A16},
pdfkeywords={cyclic vector, hypercyclic vector, weakly hypercyclic vector, typical operator}
}
\makeatletter \let\C\@undefined \makeatother 
\usepackage[msc-links,initials,backrefs]{amsrefs} 

\theoremstyle{plain}
\newtheorem{theorem}{Theorem}
\newtheorem{proposition}[theorem]{Proposition}
\newtheorem{lemma}[theorem]{Lemma}
\newtheorem{corollary}[theorem]{Corollary}

\newcommand{\lin}{\mathrm{lin}}
\renewcommand{\epsilon}{\varepsilon}
\newcommand{\N}{\mathbb{N}}
\newcommand{\C}{\mathbb{C}}
\newcommand{\R}{\mathbb{R}}

\DeclareMathOperator*{\Shift}{Shift}
\DeclareMathOperator*{\Id}{Id}

\title{Typical operators admit common cyclic vectors}
\author{Pavel Zorin-Kranich}
\address{Korteweg de Vries Instituut voor Wiskunde\\
Universiteit van Amsterdam\\
P.O.\ Box 94248\\
1090 GE Amsterdam\\
The Netherlands}
\email{zorin-kranich@uva.nl}

\begin{document}
\begin{abstract}
Given a countable dense subset $D$ of an infinite-dimensional separable Hilbert space $H$ the set of operators for which every vector in $D$ except zero is hypercyclic (weakly supercyclic) is residual for the strong (resp. weak) operator topology in the unit ball of $L(H)$ multiplied by $R>1$ (resp. $R>0$).
\end{abstract}
\maketitle

\section*{Introduction}
We show that a typical (in the Baire category sense for the strong operator topology) operator (with norm bounded by $R>1$) on a separable Hilbert space admits a prescribed countable set of hypercyclic vectors. For the weak operator topology an analogous statement holds with ``weakly supercyclic'' in place of ``hypercyclic''.

Existence proofs using the Baire category theorem have a rich history.
A classical example is the fact that a typical (in the Baire category sense) continuous functions on an interval is nowhere differentiable, due to Banach and Mazurkiewicz.
In ergodic theory it is well-known that a typical measure-preserving transformation is weakly but not strongly mixing, a result due to Halmos \cite{MR0011173} and Rohlin \cite{MR0024503}.

The following notation is used throughout the article.
\begin{itemize}
\item $H$ is a separable infinite-dimensional Hilbert space with scalar product $\langle\cdot, \cdot\rangle$.
\item $D \subset H\setminus\{0\}$ is a countable dense subset.
\item $B_R = \{ T \in L(H) : ||T|| \leq R \}$ is the closed ball of radius $R$ in the space of bounded linear operators on $H$.
\end{itemize}
For a linear operator $T$ on $H$, a vector $x\in H$ is called \emph{cyclic} if the span of the orbit $\lin \{x,Tx,T^2x,\ldots\}$ is dense in $H$, \emph{(weakly) supercyclic} if $\C \{x,Tx,T^2x,\ldots\}$ is (weakly) dense in $H$ and \emph{(weakly) hypercyclic} if the orbit itself $\{x,Tx,T^2x,\ldots\}$ is (weakly) dense in $H$.
An operator is called cyclic, resp. (weakly) supercyclic, resp. (weakly) hypercyclic if it admits a cyclic, resp. (weakly) supercyclic, resp. (weakly) hypercyclic vector.

Note that (weak) hypercyclicity implies (weak) supercyclicity that implies cyclicity.
Also, the norm of a hypercyclic operator is necessarily greater than one.
An example of a hypercyclic operator is $R$ times the left shift on $l^2(\N)$ for any $R>1$.

For convenience let us point out some known properties of hypercyclic operators.
By Birkhoff's transitivity theorem \cite{MR1555175} hypercyclicity is equivalent to topological transitivity.
Every hypercyclic operator admits a residual set of hypercyclic vectors, see for instance \cite{MR1685272}.
Furthermore, the vectors that are hypercyclic for an operator $T$ are also hypercyclic for every positive power $T^{n}$ \cite{MR1319961}.
General criteria for a family of operators to possess common hypercyclic vectors may be found in \cite{MR2317538} and \cite{MR2545665}.
The book by Bayart and Matheron \cite{MR2533318} is an excellent source of information.

\section*{Statement of results}
\begin{theorem}
\label{thm:R-strong}
Let $R>1$.
Then the subset $L_{HC}(D)$ of operators $T$ for which every $x\in D$ is hypercyclic is dense $G_\delta$ in the ball $B_R$ for the strong operator topology.
\end{theorem}

Writing $B_R =\frac{R}{2}B_2$, applying Theorem~\ref{thm:R-strong} and noting that every multiple of a hypercyclic operator is supercyclic we obtain the following result.
\begin{corollary}
\label{cor:one-strong}
Let $R>0$.
Then the subset $L_{SC}(D)$ of operators $T$ for which every $x\in D$ is supercyclic is dense $G_\delta$ in the set $B_R$ for the strong operator topology.
For every such $T$ the set of supercyclic vectors is residual in $H$ for the norm topology.
\end{corollary}
Here, the last statement is obtained by a slight modification of the standard argument for hypercyclic vectors.

A strongly typical contraction on a Hilbert space is unitarily equivalent to the left shift operator on $\ell^{2}(\N,\ell^{2}(\N))$ by a result of Eisner and Mátrai \cite[Theorem 2.10]{em2010}. In particular, it is supercyclic. Corollary~\ref{cor:one-strong} shows that strongly typical contractions even have \emph{common} supercyclic vectors.

\begin{theorem}
\label{thm:one-weak}
Let $R>0$.
Then the subset $L_{wSC}(D)$ of operators $T$ for which every $x\in D$ is weakly supercyclic is dense $G_\delta$ in the ball $B_R$ for the weak operator topology.
For every such $T$ the set of weakly supercyclic vectors is residual in $H$ for the norm topology.
\end{theorem}
By a result of Eisner \cite[Theorem 2.2]{MR2769030} a weakly typical operator in $B_1$ is also unitary.
Theorem~\ref{thm:one-weak} improves \cite[Theorem 8.26]{MR1719722} in that the typical operator is not just cyclic, but admits a \emph{prescribed} set of weakly supercyclic vectors.

Restricting attention to the set of unitary operators we recover denseness of cyclic operators in the norm topology.
\begin{theorem}
\label{thm:norm-unitary}
The set $L_{C}$ of cyclic unitary operators on $H$ is norm dense in the set of unitary operators.
\end{theorem}
Finally, a residuality result holds for the norm topology on a finite-dimensional space.
\begin{theorem}
\label{thm:norm}
Let $X$ be a finite-dimensional Banach space and $E \subset X\setminus\{0\}$ a countable subset.
Then the subset $L_C(E)$ of operators $T$ for which every $x\in E$ is cyclic is dense $G_\delta$ in the set $L(X)$ for the norm topology.
\end{theorem}

\section{Strong operator topology}
We start with an approximation result that provides the density part of the statement of Theorem~\ref{thm:R-strong}.

\begin{lemma}
\label{lemma:approx-strong-hypercyclic}
Let $x_1, \dots, x_n \in H$ be vectors with norm $1$; $R>1$; $T \in B_R$; $x \in H \setminus \{0\}$, $z \in H$ further fixed vectors and $\epsilon > 0$.
Then there exists a finite rank operator $S \in B_R$ such that $||(S-T)x_j|| < \epsilon$ for all $j$ and $z = S^N x$ for some natural number $N=N(R,\epsilon,||x||,||z||)$.
\end{lemma}
\begin{proof}
Let $P$ be the orthogonal projection onto $\lin\{ x_1, \dots, x_n\}$.
Consider the finite dimensional linear subspace $V := PH + TPH \subset H$ and denote by $P_{V}$ the orthogonal projection onto $V$.
Perturbing the vectors $x_j$ and the operator $T$ in norm and lowering $\epsilon$, we can assume $||x-P_V x|| = \gamma > 0$.
Set $e_0 := (x-P_V x) / \gamma$ and decompose $H = V \oplus \ell^2$, so that $e_0$ is the first vector of an orthonormal basis $\{e_j\}$ of $\ell^2$.

Morally, we are looking for an approximation $\tilde T$ such that $||\tilde T^N e_0||$ grows faster than $||\tilde T|_V^N||$.
Then, after a sufficient number of steps, $\tilde T^N x$ can be sent to $z$ by a tiny perturbation.

Given $M \in \N$, $\delta\in\R$ and $\tilde z \in H$ that are to be chosen later we consider the operator $\Shift_M$ on $H = V \oplus \ell^2$ that maps $e_j$ to $e_{j+1}$ if $j < M$ and to $0$ if $j\geq M$ and that is zero on $V$.
Furthermore we set $\tilde T := TP + (R+\delta) \Shift_M$ and $\phi(y):=\langle y, e_M \rangle$.
Now we make the tiny perturbation
\[
\tilde S := \tilde T + \frac1{\gamma (R + \delta)^M} \phi \otimes (\tilde z - (TP)^{M+1} x),
\]
where the tensor product of a linear form $\phi$ and a vector $v$ is the rank one operator $(\phi\otimes v)(w)=\phi(w) v$.
For every $j$ we have $\tilde S x_j = T x_j$ since $x_{j} \in V$.
Furthermore,
\begin{align*}
\tilde S^{M+1} x
&=
\tilde S^{M+1} P_{V} x + \gamma \tilde S^{M+1} e_{0}\\
&=
(TP)^{M+1} P_{V} x + \gamma (R+\delta)^{M} \tilde S e_{M}\\
&=
(TP)^{M+1} x + (\tilde z - (TP)^{M+1} x)\\
&=
\tilde z
\end{align*}
The norm of $\tilde S$ is bounded by
\[
||\tilde S||
\leq
\max\{||T||, R+\delta \} + \frac{||\tilde z - (TP)^{M+1} x||}{\gamma (R + \delta)^M}
\leq
R+\delta + \frac{||\tilde z|| + R^{M+1} ||x||}{\gamma (R + \delta)^M}.
\]
Now it is sufficient to make the following choices.
\begin{itemize}
\item $\delta$ such that $3\delta < \epsilon$ and $R + \delta > 1 + 3 \delta / R$,
\item $M$ such that $\gamma (R + \delta)^M > R^{M+1} ||x|| / \delta$ and $\gamma (R + \delta)^M > ||z|| (1 + 3 \delta / R)^{M+1} / \delta$,
\item $\tilde z = z (1 + 3\delta / R)^{M+1}$.
\end{itemize}

Then $|| \tilde S || \leq R + 3 \delta$ and thus $S := \frac{R}{R + 3 \delta} \tilde S$ has norm less than or equal to $R$.
It also satisfies $S^N x = z$ with $N=M+1$ and
\[
||(S-T) x_j|| = ||(S - \tilde S) x_j|| = \frac{3 \delta}{R + 3\delta} ||\tilde S x_j|| \leq 3 \delta ||x_j|| < \epsilon.
\qedhere
\]
\end{proof}

Now we use the Baire category theorem to amplify this result.
\begin{proposition}\label{prop:residual-strong-hypercyclic}
Let $R>1$. Given a vector $x \in H \setminus \{0\}$, the set $L_{HC}(x)$ of operators on $H$ for which $x$ is a hypercyclic vector is dense $G_\delta$ in $B_R$ for the strong operator topology.
\end{proposition}
\begin{proof}
Let $(O_j)_{j=1}^\infty$ be a countable base of the norm topology on $H$. Then
\[
L_{HC}(x) = \cap_j M_j, \text{ where } M_j = \cup_n \{ T: T^n x \in O_j \}.
\]
The sets $M_j$ are strongly dense in $B_{R}$ by Lemma~\ref{lemma:approx-strong-hypercyclic} with $z \in O_j$ and open since $T^{n}x \in O_j$ implies $S^{n}x \in O_{j}$ for all $S$ in a strong neighborhood of $T$ associated with the seminorms generated by $x, \dots, T^{N-1} x$ (note that we only consider a bounded set of operators).

Hence the set $L_{HC}(x)$ is $G_\delta$ in $B_R$.
The Baire category theorem implies that it is dense.
\end{proof}

\begin{proof}[Proof of Theorem~\ref{thm:R-strong}]
Since $D$ is countable, $L_{HC}(D) = \cap_{x \in D} L_{HC}(x)$ is dense $G_\delta$ by Proposition \ref{prop:residual-strong-hypercyclic} and the Baire category theorem.
\end{proof}

Note that our proof still works if we replace $H$ by an arbitrary Banach space that admits projections with norm arbitrarily close to one onto every finite-dimensional subspace.

We now prove the last statement in Corollary~\ref{cor:one-strong}.
Let $(O_j)_{j=1}^\infty$ be a countable base of the norm topology on $H$.
The set of supercyclic vectors of $T$ is given by
\[
\cap_{j=1}^\infty M_{j}, \text{ where } M_{j} = \cup_{n\in\N, a\in\C} \{x\in H: a T^{n}x \in O_{j} \}.
\]
Each $M_{j}$ is open because $T$ is continuous and dense because it contains $D$.
The Baire category theorem concludes the proof.

\section{Weak operator topology}
In principle, the switch to a coarser operator topology makes the approximation easier.
However, we need additional information in order to upgrade weak to strong convergence in the amplification process.
\begin{lemma}
\label{lemma:approx-weak-hypercyclic}
Let $x_1, \dots, x_n, y_1, \dots, y_n \in H$ be fixed vectors with norm $1$; $R>1$, $T \in B_R$, $x \in H\setminus\{0\}$, $z \in H$ further fixed vectors and $\epsilon > 0$.
Then there exists a finite rank operator $S \in B_R$ and a natural number $N=N(R,\epsilon,||x||,||z||)$ such that $|\langle (S-T)x_j, y_j \rangle| < \epsilon$ and $\langle z - S^N x, y_j \rangle = 0$ for every $j$ and $||S^{l+1} x|| = R ||S^l x||$ for every $l < N$.
\end{lemma}
 \begin{proof}
By Lemma~\ref{lemma:approx-strong-hypercyclic} we can assume that $z = T^N x$, that $T(H) \subset V$ and that $T$ vanishes on $V^\perp$, where $V$ is a finite-dimensional subspace.

Only the property $||S^{l+1} x|| = R ||S^l x||$ for $l < N$ is missing.

Observe that $[x,y]:=\langle (R^{2}\Id - T^* T)x, y \rangle$ is a positive semidefinite sesquilinear form on $V$.
Let $\tilde S_{id} : V \to W := (V, [\cdot,\cdot]) / \{x : [x,x]=0\}$ be the canonical projection onto the quotient of $V$ by the $[\cdot,\cdot]$-zero vectors.
Note that $(W,[\cdot,\cdot])$ is a finite-dimensional pre-Hilbert space, hence complete.
The operator $T \oplus \tilde S_{id} : V \to T(V) \oplus W$ is $R$ times an isometry since $\langle (T\oplus \tilde S_{id}) x, (T\oplus \tilde S_{id}) y \rangle_{T(V) \oplus W} = \langle Tx,Ty\rangle_{H} + [x,y] = R^{2}\langle x,y\rangle$.

Now select isometric copies $W_1, W_2, \dots, W_N \subset H$ of $W$ orthogonal to $V$, all of $y_i$, and to each other.
Let $S_{id} : V \to W_1$ be equivalent to $\tilde S_{id}$ (i.e.\ $S_{id} = U \tilde S_{id}$ with a unitary $U$) and $S_l : W_{l-1} \to W_l$ be $R$ times unitary for $2 \leq l \leq N$.
Extend all the $S_l$ to $H$ by letting them be $0$ on the subspaces orthogonal to their respective domains of definition.
Then $S:=T + S_{id} + S_2 + \dots + S_N$ still approximates $T$ weakly (since its values differ from those of $T$ only by vectors orthogonal to all of $y_j$) and $||S^{l+1} x|| = R ||S^l x||$ for $l < N$.
\end{proof}

\begin{corollary}
\label{cor:approx-weak-supercyclic}
Let $x_1, \dots, x_n, y_1, \dots, y_n \in H$ be fixed vectors with norm $1$; $T \in B_1$, $x \in H\setminus\{0\}$, $z \in H$ further fixed vectors and $\epsilon > 0$.
Then there exists a finite rank operator $S \in B_1$, a number $a \in \R$ and a natural number $N=N(\epsilon)$ such that $|\langle (S-T)x_j, y_j \rangle| < \epsilon$ and $\langle z - a S^N x, y_j \rangle = 0$ for every $j$ and $||S^{l+1} x|| = ||S^l x||$ for every $l < N$.
\end{corollary}
\begin{proof}
Apply Lemma~\ref{lemma:approx-weak-hypercyclic} with a small $R>1$ and divide the resulting approximation $S$ by $R$.
\end{proof}

\begin{proposition}
\label{prop:residual-weak-hypercyclic}
Given a vector $x \in H \setminus \{0\}$, the set $L_{wSC}(x)$ of operators on $H$ for which $x$ is weakly supercyclic is dense $G_\delta$ in $B_1$ for the weak operator topology.
\end{proposition}
\begin{proof}
Let $(O_j)_{j=1}^\infty$ be a countable base of the weak topology on the open unit ball of $H$. Then
\[
L_{wSC}(x) = \cap_j M_j, \text{ where } M_j = \cup_{n\in\N,a\in\C} \{ T \in B_1 : a T^n x \in O_j \}.
\]

By the Baire category theorem it is sufficient to show that in the weak operator topology the interior of each $M_j$ is dense.
For this end assume that $O_j$ is the weak neighborhood of $z$ determined by $y_1,\dots, y_m$ and $\epsilon'$. 
We claim that for any $T$, any collection $x_1, \dots, x_n, y_{m+1}, \dots, y_n$, and any $\epsilon < \epsilon'$, the approximation $S$ given by Corollary~\ref{cor:approx-weak-supercyclic} is in the interior of $M_j$.
Indeed, assume the opposite.
Then there exists a sequence $\{S_k\}_{k=1}^\infty\subset B_{R} \setminus M_j$ converging to $S$ weakly.
Since $||S (S^{l} x)||=||S^{l} x||$ for every $l=0,\dots,N-1$ and $||S_{k}||\leq 1$, we see that
\begin{multline*}
||(S_k-S)(S^{l}x)||^2
=
||S_k (S^{l}x)||^{2} + ||S (S^{l}x)||^{2} - 2 \Re \langle S_{k}(S^{l}x), S(S^{l}x) \rangle\\
\leq
2 ||(S^{l}x)||^{2} - 2 \Re \langle S_{k}(S^{l}x), S(S^{l}x) \rangle
\to
2 ||(S^{l}x)||^{2} - 2 \Re \langle S(S^{l}x), S(S^{l}x) \rangle
=
0
\end{multline*}
as $k \to \infty$, cf.~\cite[Problems 20 and 21]{MR0208368}.
Induction on $l$ using the estimate
\[
\|S_k^l x-S^l x\| 
\leq
\|S_k\|\cdot\|S_k^{l-1}x-S^{l-1}x\| + \|S_k(S^{l-1} x) - S (S^{l-1}x)\|
\] 
shows that $\lim_{k\to\infty}\|S_k^l x-S^l x\|=0$ for $l=0,\dots, N$, and in particular $\lim_{k\to\infty} S_k^N x = S^N x$.
Thus all but finitely many $S_k$ belong to $M_j$, a contradiction.
\end{proof}

\begin{proof}[Proof of Theorem~\ref{thm:one-weak}]
It suffices to consider the case $R=1$.
Since $D$ is countable, $L_{wSC}(D) = \cap_{x \in D} L_{wSC}(x)$ is dense $G_\delta$ by Proposition \ref{prop:residual-weak-hypercyclic} and the Baire category theorem.

To show the last assertion of the theorem let $T \in L_{wSC}(D)$ and $(O_j)_{j=1}^\infty$ be a countable base of the weak topology on the open unit ball of $H$.
The set of weakly supercyclic vectors of $T$ is given by
\[
\cap_{j=1}^\infty M_{j}, \text{ where } M_{j} = \cup_{n\in\N,a\in\C} \{x\in H: aT^{n}x \in O_{j} \}.
\]
Each $M_{j}$ is open because $O_{j}$ is open in the norm topology and $T$ is continuous, and dense because $D \subset M_{j}$ by assumption. The Baire category theorem concludes the proof.
\end{proof}

\section{Norm topology}
Approximation of operators in the norm topology is harder than in the strong or in the weak operator topology.
We begin with the unitary case, where the spectral theorem reduces the task to a computation.

\begin{proof}[Proof of Theorem~\ref{thm:norm-unitary}]
Let $U$ be a unitary operator on $H$.
Using the spectral theorem we can approximate $U$ by unitary operators such that the Hilbert space $H$ admits an orthonormal basis $\{ e_i \}_{i = 0,1, \dots}$ of eigenvectors to pairwise distinct eigenvalues $\lambda_i$, cf.\ \cite[Prop. 2.1]{MR2431104}.
Thus we may assume that $U$ enjoys this property. We claim that $U$ is cyclic.

Denote by $\Lambda_N$ the Vandermonde matrix generated by $\lambda_0, \dots, \lambda_N$.
In the finite-dimensional case any vector that is not orthogonal to any of $e_{i}$'s is cyclic since the Vandermonde matrix $\Lambda_{\dim H-1}$ is invertible.

In the infinite-dimensional case let $a_N := (N+1)! |\det \Lambda_N|^{-1}$ and define a sequence $(b_{M})$ inductively starting with $b_0 = 1$ and setting
\[
b_{M+1} := \left( \min_{N \leq M} \frac12 \left( \frac{|b_N|^2}{a_N^2} - \sum\limits_{j = N+1}^M |b_j|^2 \right) \right)^{1/2}
\]
with the usual convention for the sum to be zero if the lower bound of summation is greater than the upper bound.
The sequence obtained in this way is strictly positive and rapidly decreasing in the sense that, for every $N$,
\[
\sum\limits_{j= N+1}^\infty b_j^2 \leq \frac {b_N^2}{a_N^2}.
\]
In particular, it is quadratically summable.
We claim that $y := \sum_{j=0}^\infty b_j e_j$ is a cyclic vector for $U$.
It is sufficient to approximate every $e_j$ by finite linear combinations of $U^k y$.
Observe that the partial sums $y_N = \sum_{j=0}^N b_j e_j$ satisfy
\[
(y_N, U y_N, \dots, U^N y_N) = (b_0 e_0, \dots, b_N e_N) \Lambda_N
\]
for every $N$. Hence, for every $k \leq N$, we have
\[
e_k = ((y_N, U y_N, \dots, U^N y_N) \Lambda_N^{-1})_k / b_k.
\]
Since the entries of $\Lambda_N$ have absolute value one, the elements of $\Lambda_N^{-1}$ are bounded by $N! / |\det \Lambda_N|$. Therefore
\begin{align*}
||e_k - (&(y, U y, \dots, U^N y) \Lambda_N^{-1})_k / b_k||^2\\
&=
|| \sum_{j=0}^N U^j (y_N - y) (\Lambda_N^{-1})_{jk} ||^2 / b_k^2\\
&\leq (N+1)^2 (N!)^2 |\det\Lambda_N|^{-2} ||y-y_N||^2 / b_k^2\\
&= a_N^2 \sum\limits_{j= N+1}^\infty b_j^2 / b_k^2\\
&\leq b_N^2 / b_k^2 \to 0 \text{ as } N \to \infty,
\end{align*}
and we are done because $((y, U y, \dots, U^N y) \Lambda_N^{-1})_k / b_k \in \lin\{y, Uy, \dots \}$.
\end{proof}

In the finite-dimensional case we merely need the first few elements of the orbit to be linearly independent to obtain cyclicity.
As the following lemma shows we can even make the full orbit linearly independent, at least in the infinite-dimensional case, though there is no guarantee that it shall span the whole space.

\begin{lemma}
Let $X$ be an infinite dimensional Banach space, $T \in L(X)$ a contraction, $x \in X\setminus\{0\}$ and $\epsilon > 0$.
Then there exists a contraction $S$ such that $||T-S|| < \epsilon$ and $x, Sx, S^2x, \dots$ are linearly independent.
\end{lemma}
\begin{proof}
Replacing $T$ by $T(1-\epsilon/2)$ and $\epsilon$ by $\epsilon/2$ we may assume that $||T||\leq 1-\epsilon$.

Let $T_0 := T$ and construct a sequence of operators $T_k$ inductively.
Given $T_k$, define $V_k := \lin \{T_k^0 x, \dots, T_k^{k-1} x\}$, assume $T_k^k x \not\in V_k$ and choose a $\psi_k \in X'$ such that $\psi_k|_{V_k} = 0$, $\psi_k (T_k^k x) = 1$.
Choose $w_k \not\in V_k + \lin \{ T_k^k x, T_k^{k+1} x \}$ with $||w_k|| < \epsilon/(2^{k+1} ||\psi_k||)$ and set $T_{k+1} := T_k + \psi_k \otimes w_k$.

Then $T_{k+1}^l x = T_k^l x$ for $l \leq k$ and $T_{k+1}^{k+1} x \not\in V_{k}+\lin\{T_{k}^{k}\} = V_{k+1}$.
The contractions $T_k$ converge in norm to a contraction $S$ such that $||T - S|| < \epsilon$.
The orbit of $x$ under $S$ is linearly independent by construction.
\end{proof}
In the finite-dimensional case this construction breaks down because
\[
V_k + \lin \{ T_k^k x, T_k^{k+1} x \} = X
\]
for some $k$, so that $T_{k}$ is in fact a cyclic contraction.
Furthermore, any operator in a neighborhood of $T_{k}$ is also cyclic.
Hence the subset $L_C(x) \subset L(X)$ of operators for which $x$ is cyclic is open and dense in $L(X)$ equipped with the norm topology.
An application of the Baire category theorem gives Theorem~\ref{thm:norm}.

\section*{Acknowledgement}
The author would like to thank Tanja Eisner for bringing the matter to his attention and for helpful discussions.


\begin{bibdiv}
\begin{biblist}

\bib{MR1319961}{article}{
      author={Ansari, Shamim~I.},
       title={Hypercyclic and cyclic vectors},
        date={1995},
        ISSN={0022-1236},
     journal={J. Funct. Anal.},
      volume={128},
      number={2},
       pages={374\ndash 383},
         url={http://dx.doi.org/10.1006/jfan.1995.1036},
      review={\MR{MR1319961 (96h:47002)}},
}

\bib{MR2317538}{article}{
      author={Bayart, F.},
      author={Matheron, {\'E}.},
       title={How to get common universal vectors},
        date={2007},
        ISSN={0022-2518},
     journal={Indiana Univ. Math. J.},
      volume={56},
      number={2},
       pages={553\ndash 580},
         url={http://dx.doi.org/10.1512/iumj.2007.56.2863},
      review={\MR{MR2317538 (2008f:47008)}},
}

\bib{MR2533318}{book}{
      author={Bayart, Fr{\'e}d{\'e}ric},
      author={Matheron, {\'E}tienne},
       title={Dynamics of linear operators},
      series={Cambridge Tracts in Mathematics},
   publisher={Cambridge University Press},
     address={Cambridge},
        date={2009},
      volume={179},
        ISBN={978-0-521-51496-5},
         url={http://dx.doi.org/10.1017/CBO9780511581113},
      review={\MR{MR2533318}},
}

\bib{MR1555175}{article}{
      author={Birkhoff, George~D.},
       title={Surface transformations and their dynamical applications},
        date={1922},
        ISSN={0001-5962},
     journal={Acta Math.},
      volume={43},
      number={1},
       pages={1\ndash 119},
         url={http://dx.doi.org/10.1007/BF02401754},
      review={\MR{MR1555175}},
}

\bib{MR2769030}{article}{
      author={Eisner, Tanja},
       title={A ``typical'' contraction is unitary},
        date={2010},
        ISSN={0013-8584},
     journal={Enseign. Math. (2)},
      volume={56},
      number={3-4},
       pages={403\ndash 410},
         url={http://www.fa.uni-tuebingen.de/members/talo/unitary.pdf},
      review={\MR{2769030}},
}

\bib{em2010}{unpublished}{
      author={Eisner, Tanja},
      author={M{\'a}trai, Tam{\'a}s},
       title={On typical properties of {H}ilbert space operators},
        date={2010},
         url={http://arxiv.org/abs/1008.3326},
        note={Preprint, see \url{http://arxiv.org/abs/1008.3326}},
}

\bib{MR2431104}{article}{
      author={Eisner, Tanja},
      author={Ser{\'e}ny, Andr{\'a}s},
       title={Category theorems for stable operators on {H}ilbert spaces},
        date={2008},
        ISSN={0001-6969},
     journal={Acta Sci. Math. (Szeged)},
      volume={74},
      number={1-2},
       pages={259\ndash 270},
      review={\MR{2431104 (2009i:47027)}},
}

\bib{MR1685272}{article}{
      author={Grosse-Erdmann, Karl-Goswin},
       title={Universal families and hypercyclic operators},
        date={1999},
        ISSN={0273-0979},
     journal={Bull. Amer. Math. Soc. (N.S.)},
      volume={36},
      number={3},
       pages={345\ndash 381},
         url={http://dx.doi.org/10.1090/S0273-0979-99-00788-0},
      review={\MR{MR1685272 (2000c:47001)}},
}

\bib{MR0011173}{article}{
      author={Halmos, Paul~R.},
       title={In general a measure preserving transformation is mixing},
        date={1944},
        ISSN={0003-486X},
     journal={Ann. of Math. (2)},
      volume={45},
       pages={786\ndash 792},
      review={\MR{0011173 (6,131d)}},
}

\bib{MR0208368}{book}{
      author={Halmos, Paul~R.},
       title={A {H}ilbert space problem book},
   publisher={D. Van Nostrand Co., Inc., Princeton, N.J.-Toronto, Ont.-London},
        date={1967},
      review={\MR{0208368 (34 \#8178)}},
}

\bib{MR1719722}{book}{
      author={Nadkarni, M.~G.},
       title={Spectral theory of dynamical systems},
      series={Birkh\"auser Advanced Texts: Basler Lehrb\"ucher. [Birkh\"auser
  Advanced Texts: Basel Textbooks]},
   publisher={Birkh\"auser Verlag},
     address={Basel},
        date={1998},
        ISBN={3-7643-5817-3},
      review={\MR{MR1719722 (2001d:37001)}},
}

\bib{MR0024503}{article}{
      author={Rohlin, V.},
       title={A ``general'' measure-preserving transformation is not mixing},
        date={1948},
     journal={Doklady Akad. Nauk SSSR (N.S.)},
      volume={60},
       pages={349\ndash 351},
      review={\MR{MR0024503 (9,504d)}},
}

\bib{MR2545665}{article}{
      author={Sanders, Rebecca},
       title={Common hypercyclic vectors and the hypercyclicity criterion},
        date={2009},
        ISSN={0378-620X},
     journal={Integral Equations Operator Theory},
      volume={65},
      number={1},
       pages={131\ndash 149},
         url={http://dx.doi.org/10.1007/s00020-009-1711-0},
      review={\MR{MR2545665 (2010i:47019)}},
}

\end{biblist}
\end{bibdiv}
\end{document}